\documentclass{amsart}

\usepackage[all]{xy}
\usepackage{hyperref}
\usepackage{amssymb}
\usepackage{mathdots}
\hypersetup{
    colorlinks,%
    citecolor=blue,%
    filecolor=blue,%
    linkcolor=blue,%
    urlcolor=blue
}
\usepackage{booktabs}
\usepackage{mathrsfs}
\usepackage{todonotes}
\usepackage{bbm}
\usepackage[lite]{amsrefs} 
\usepackage{longtable}

\newtheorem{introthm}{Theorem}
\newtheorem{intropro}{Proposition}

\newtheorem{theorem}{Theorem}[section]
\newtheorem{lemma}[theorem]{Lemma}
\newtheorem{proposition}[theorem]{Proposition}
\newtheorem{corollary}[theorem]{Corollary}
\theoremstyle{definition}

\newtheorem{example}[theorem]{Example}

\newtheorem{construction}[theorem]{Construction}

\newtheorem{remark}[theorem]{Remark}
\theoremstyle{remark}

\numberwithin{equation}{section}

\def\cc{{\mathbb K}}
\def\kk{k}
\def\kb{{\bar k}}

\def\pp{{\mathbb P}}

\def\Osh{{\mathcal O}}
\def\R{{\mathcal R}}

\def\Cl{\operatorname{Cl}}

\def\Gal{\operatorname{Gal}}

\def\id{\operatorname{id}}

\def\WDiv{\operatorname{WDiv}}
\def\PDiv{\operatorname{PDiv}}
\def\Spec{\operatorname{Spec}}

\def\cl{\operatorname{cl}}

\def\Frac{\operatorname{Frac}}

\def\Hom{\operatorname{Hom}}

\def\b1{\mathbbm{1}}

\renewcommand{\div}{\operatorname{div}}

\begin{document}

\title{Cox ring of the generic fiber}

\author[A.~Laface]{Antonio Laface}
\address{
Departamento de Matem\'atica,
Universidad de Concepci\'on,
Casilla 160-C,
Concepci\'on, Chile}
\email{alaface@udec.cl}

\author[L.~Ugaglia]{Luca Ugaglia}
\address{
Dipartimento di Matematica e Informatica,
Universit\`a degli studi di Palermo,
Via Archirafi 34,
90123 Palermo, Italy}
\email{luca.ugaglia@unipa.it}


\thanks{
The first author was partially supported 
by Proyecto FONDECYT Regular N. 1190777.
The second author is member of INdAM - GNSAGA.
Both authors have been partially supported 
by project Anillo ACT 1415 PIA Conicyt.
}


\begin{abstract}
Given a surjective morphism $\pi\colon X\to Y$ of 
normal varieties satisfying some regularity hypotheses
we prove how to recover a Cox ring of the generic fiber 
of $\pi$ from the Cox ring of $X$.
As a corollary we show that in some cases
it is also possible to recover the Cox ring of a very 
general fiber, and finally we give an application in 
the case of the blowing-up of a toric fiber space.
\end{abstract}

\maketitle

\section*{Introduction}
Let $X$ be a normal variety defined over
an algebraically closed field $\mathbb K$ of 
characteristic zero.
If the divisor class group ${\rm Cl}(X)$ of $X$ 
is finitely generated and $\cc [X]^* = 
\cc^*$, i.e. the only global regular
invertible functions of $X$ are constants,
the {\em Cox sheaf} of $X$ can be defined as
(see~\cite{ADHL})
\[
 \mathcal R 
 := 
 \bigoplus_{[D]\in {\rm Cl}(X)}\mathcal O_X(D),
\]
while its {\em Cox ring} $\mathcal R(X)$ 
is the ring of global sections $\Gamma(X,\mathcal R)$.
Given a morphism $\pi\colon X\to Y$ of normal 
varieties defined over $\cc$, possible
relations between the Cox rings of $X$ and $Y$
have been recently studied in many cases 
(see for instance~\cites{ok,Bak,HaS,AHL,HKL}).
On the contrary, to our knowledge there are 
no results concerning relations with the Cox ring of the
fibers of $\pi$.
Trying to fill this gap, in the present paper we consider 
the problem of determining the Cox ring of the generic 
fiber $X_\eta$ (and in some cases also
the Cox ring of the very general fiber)
of $\pi$ from the Cox ring of $X$ and from the 
{\em vertical classes} of $\pi$, i.e. classes of divisors 
whose image in $Y$ is not dense.
Observe that since $X_\eta$ is defined over a 
non closed field (isomorphic to the function field $\cc(Y)$),
we need to define a Cox ring for $X_\eta$ following~\cite{dp}.

In order to describe our results let us 
denote by $\Cl_\pi(X)$ the subgroup 
of $\Cl(X)$ generated by classes 
of vertical divisors, or equivalently
the kernel of the surjection 
$\Cl(X)\to\Cl(X_\eta)$, induced by the
pull-back of the natural morphism $\imath\colon
X_\eta\to X$.
If we denote by $\R_\pi(X)$ the localization
of $\R(X)$ by the multiplicative subsystem
generated by the non-zero homogeneous 
elements $f\in\R(X)_w$, with $w\in \Cl_\pi(X)$,
and by $\Frac_0(\R(X))$
the field of degree zero 
homogeneous fractions on $\R(X)$,
the following holds.
\begin{intropro}
\label{main:propo}
The image of the homomorphism
$\cc(Y)\to \Frac_0(\R(X))$
induced by the pullback is $\R_\pi(X)_0$, 
the subset of degree zero homogeneous 
elements of $\R_\pi(X)$.
\end{intropro}
One consequence of the proposition above
is that $\R_\pi(X)$ has a structure of
$\cc(Y)$-algebra. On the other hand, 
since the generic fiber $X_\eta$ 
is defined over a field $\kk$, isomorphic to 
$\cc(Y)$, following~\cite{dp} we can construct 
a Cox ring $\R(X_\eta)$ which has the 
structure of $\cc(Y)$-algebra too.
Our main result is a description 
of the relation existing between these 
two algebras. In particular the latter turns out to 
be a quotient of the former,
and the precise result is the content of
the following.
\begin{introthm}
\label{main:thm}
Let $\pi\colon X\to Y$ be a proper surjective 
morphism of normal varieties 
having only constant invertible global sections, 
such that $\Cl(X)$ is finitely generated, $\Cl_\pi(X)$ torsion free,
and the very general fiber of $\pi$ is irreducible.
Then there exists a Cox ring 
$\R(X_\eta)$ of the generic fiber $X_\eta$ such that
the canonical morphism $\imath\colon X_\eta\to X$ 
induces an isomorphism of $\Cl(X_\eta)$-graded
$\cc(Y)$-algebras
\[
 \R_\pi(X)/\langle 1-u(w)\, :\, w\in\Cl_\pi(X)\rangle
 \to\R(X_\eta),
\]
where $u\colon \Cl_\pi(X)\to \R_\pi(X)^*$
is any homomorphism satisfying 
$u(w) \in \R_\pi(X)_{-w}^*$ for each 
$w$.
\end{introthm}
Let us suppose in addition that the class group of the
geometric generic fiber $X_\eta\times_\kk\kb$ 
is isomorphic to $\Cl(X_\eta)$. 
We will show (see Corollary~\ref{main:cor})
that in this case it is possible to 
combine the theorem above
with the results of~\cite{vi} in order to
recover the Cox ring of a very general
fiber of $\pi$ from the Cox ring of $X$. 
A direct consequence is that finite
generation for the Cox ring of $X$
implies finite generation for 
the Cox ring of the very general fiber.
Applying these results to the blowing-up 
of a toric fiber space along a section, 
we will finally produce new examples of 
varieties with non-finitely generated Cox ring.

The paper is structured as follows.
In Section~\ref{sec-1} we first recall the definition
of Cox sheaf and Cox ring for a variety defined
over a closed field, and then, after remembering 
some facts about varieties defined over a 
(not necessarily closed) perfect field,
following~\cite{dp} we construct a Cox sheaf 
for such varieties.
In Section~\ref{sec:fib} we collect some results
about the generic fiber $X_\eta$ of a proper surjective
morphism $\pi\colon X\to Y$, whose very general fiber is 
irreducible. Section~\ref{sec:div} contains
the proof of Proposition~\ref{main:propo} 
and some lemmas that we are going to use 
in Section~\ref{sec:prf}, where we prove 
Theorem~\ref{main:thm}. In Section~\ref{sec:cor}
we consider the very general fiber and 
in the last section we apply the results above
to the blowing up of toric fiber spaces
along a section.

\subsection*{Acknowledgements}
We would like to thank Prof. Ulrich Derenthal for many useful
comments.

\section{Preliminaries}
\label{sec-1}

In this section we first recall
the definition of Cox sheaf and Cox 
ring in the case of a variety defined
over an algebraically closed field
(see~\cite{ADHL}). 
Then, after recalling some known facts
about algebraic varieties defined 
over a (non necessarily closed) perfect field 
we construct a suitable {\em Cox sheaf
of type} $\lambda$ for such varieties, 
according to~\cite{dp}*{Definition 2.2}.

\subsection{Algebraically closed fields}

\begin{construction}(see~\cite{ADHL})
\label{con:closed}
Let $X$ be a normal variety defined
over a closed field $\cc$, such that
$\cc [X]^* = \cc^*$ and $\Cl(X)$ is finitely generated.
Let $K$ be a finitely generated 
subgroup of $\WDiv(X)$ such that
the class map $\cl\colon K\to \Cl(X)$ 
is onto. The {\em sheaf of divisorial algebras} 
associated to $K$ and its global sections are
\[
\mathcal S = \bigoplus_{D\in K}\Osh_X(D)
\quad
{\rm and}
\quad
\Gamma(X,\mathcal S) 
= \bigoplus_{D\in K} \Gamma(X,\Osh_X(D))
\]
respectively. In the rest of the paper, when
we need to keep trace of the group $K$, 
we will use the notation $\mathcal S_K$
instead of $\mathcal S$.
We will also denote by $\Gamma(X,\mathcal S)_D$
the degree $D$ part of the ring of global sections
of $\mathcal S$. 
Let us consider now the kernel $K^0\subseteq K$ 
of  the class map and  let $\mathcal{X}: K^0\to\cc(X)^*$ 
be a homomorphism of groups 
such that $\div\circ\mathcal{X}=\id$.
Let $\mathcal{I}$ be the ideal sheaf of
$\mathcal{S}$, locally generated by
sections of the form $1-\mathcal{X}(D)$, where
$D\in K^0$. The quotient 
$\mathcal R := \mathcal{S}/\mathcal{I}$ turns out to
be a sheaf (see~\cite{ADHL}*{Lemma~1.4.3.5}) and 
it is called a {\em Cox sheaf} for $X$.
The {\em Cox ring} $\mathcal R(X)$ of $X$
can be defined as the ring of global sections 
of $\mathcal R(X)$, or equivalently 
\begin{equation}
\label{def:cox}
\mathcal R(X) = 
\frac{\Gamma(X,\mathcal S)}
{\Gamma(X,\mathcal I)}.
\end{equation}
\end{construction}

\subsection{Non closed fields}
Let us recall now some facts 
about an algebraic variety $X$, 
defined over a perfect field $\kk$ (see for 
instance~\cite{HS}*{$\S$~A.1 and $\S$~A.2}).
In what follows we will denote by $X_\kb$ the 
base change of $X$ over the algebraic 
closure $\kb$ of $\kk$.
From now on we assume that any variety 
$X$ has only constant invertible global sections, i.e. 
\begin{equation}
\label{assump}
\kb[X]^* = \kb^*, 
\end{equation}
where $\kb[X]$ denotes the ring of global sections 
of the structure sheaf of $X$.
Let us denote by $G = \Gal(\kb/\kk)$ the absolute Galois 
group of $\kk$ and let
\begin{eqnarray*}
\WDiv(X) &  := &
 \left\{D\in\WDiv(X_\kb) : \sigma(D)=
 D\text{ for any }\sigma\in G\right\}
\end{eqnarray*}
be the group of $G$-invariant Weil divisors 
of $X_\kb$. We will denote by $\PDiv(X)$ the 
subgroup of $\WDiv(X)$ consisting of principal
divisors of the form $\div(f)$, with
$f\in\kk(X)$. By~\cite{HS}*{Proposition A.2.2.10 (ii)}
the equality $\PDiv(X)
=\WDiv(X) \cap\PDiv(X_{\bar\kk})$
holds (observe that the hypothesis 
in the cited proposition asks for $X$ to be
projective, but it actually only makes
use of the weaker condition~\eqref{assump}).
Thus, if we denote by $\Cl(X)$ the
quotient group $\WDiv(X)/\PDiv(X)$,
we get inclusions
\begin{equation}
\label{incl}
 \Cl(X)\subseteq \Cl(X_\kb)^G
 \subseteq\Cl(X_\kb).
\end{equation}
Given a divisor $D\in\WDiv(X)$ and 
a Zariski open subset $U$ of $X$, the 
space of sections $\Osh_{X_\kb}(D)(U_\kb)$
is a $\kb$ vector space acted by $G$,
since both $U$ and $X$ are
defined over $\kk$, and thus it is a 
$G$-module. 
Observe that a $G$-invariant element $f\in
\Osh_{X_\kb}(D)(U_\kb)^G$
is a rational function of $X_\kb$,
which is defined over $\kk$ (see for 
instance~\cite{sil}*{Exercise  1.12}).
If we set
\begin{equation}
 \label{rr}
 \Osh_X(D)(U)
 :=
 \Osh_{X_\kb}(D)(U_\kb)^G,
\end{equation}
by~\cite{HS}*{Proposition A.2.2.10 (i)} we have that
the $\kb$ vector space $\Osh_X(D)(U)\otimes_\kk\kb$
is isomorphic to $\Osh_{X_\kb}(D)(U_\kb)$.

We are now able to construct 
a sheaf $\R$ of $\Osh_X$-algebras 
which turns out to be a {\em Cox sheaf
of type $\lambda$} according to
~\cite{dp}*{Definitions 2.2 and 3.3}.

\begin{construction}
\label{con:nonclosed}
Let $X$ be a variety defined over a perfect field $\kk$,
satisfying~\eqref{assump}.
Let us suppose that $\Cl(X)$ is finitely 
generated and let $K$ be a finitely generated 
subgroup of $\WDiv(X)$ whose
image via the class map $\cl\colon K\to \Cl(X_\kb)$ 
is $\Cl(X)$.
Let us consider the $K$-graded sheaf of
$\mathcal \Osh_X$-algebras
\[
 \mathcal{S}
 =
 \bigoplus_{D\in K} \Osh_X(D),
\]
where $\Osh_X(D)$ is the sheaf defined in~\eqref{rr}.
Denote by $K^0\subseteq K$ the kernel of 
the class map and  let $\mathcal{X}: K^0\to\kk(X)^*$ be a 
homomorphism of groups such that
$\div\circ\mathcal{X}=\id$ (such a $\mathcal{X}$ 
exists again by~\cite{HS}*{Proposition A.2.2.10 (ii)}).
Let $\mathcal{I}$ be the ideal sheaf of
$\mathcal{S}$ locally generated by
sections of the form $1-\mathcal{X}(D)$, where
$D\in K^0$. Denote by $\mathcal R$ 
the presheaf $\mathcal{S}/\mathcal{I}$ and by 
$\pi\colon\mathcal S\to \R$ the quotient map.
\end{construction}

\begin{proposition}
The presheaf $\mathcal R$ defined above is 
a Cox sheaf of type $\lambda$ 
(see~\cite{dp}*{Definition 3.3})
where $\lambda\colon \Cl(X)\to\Cl(X_\kb)$
is the inclusion.
\end{proposition}
\begin{proof}
By construction $K/K^0$ is isomorphic
to the group $\Cl(X)$, so that grading
over the former is equivalent to grading
over the latter.
Consider the sheaf of divisorial algebras
\[
 \bar{\mathcal S} = \bigoplus_{D\in K}\Osh_{X_\kb}(D)
\]
together with the $G$-invariant character 
$\mathcal X\colon K^0\to \kk(X)^*\subseteq\kb(X)^*$,
where $G=\Gal(\kb/\kk)$ as before,
and let $\bar{\mathcal I}$ be the ideal sheaf of
$\bar{\mathcal S}$ defined by $\mathcal X$.
Let us denote by $\phi\colon X_\kb\to X$ the 
base change map.
According to the proof of~\cite{dp}*{Proposition 3.19}
the quotient sheaf $\bar{\mathcal R} =
\bar{\mathcal S}/\bar{\mathcal I}$ is a 
$G$-equivariant Cox sheaf of type 
$\lambda$ and the push forward of the 
sheaf of invariants $\phi_*\bar{\mathcal R}^G$ is a 
Cox sheaf of $X$ of type $\lambda$.
Given an open subset $U\subseteq X$
and a divisor $D\in K$ the following holds
\begin{align*}
 (\phi_*\bar{\mathcal R}_{[D]}^G)(U)
 & =
 (\bar{\mathcal R}_{[D]}(U_\kb))^G\\[2pt]
 & =
 ((\bar{\mathcal S}/\bar{\mathcal I})_{[D]}(U_\kb))^G\\[2pt]
 & \simeq
 (\bar{\mathcal S}_D(U_\kb)/\bar{\mathcal I}_D(U_\kb))^G\\[2pt]
 & \simeq
 \mathcal S_D(U)/\mathcal I_D(U),
\end{align*}
where the first isomorphism is by
~\cite{dp}*{Construction 2.7}, 
while the second one is by Lemma~\ref{quot}
and equality~\eqref{rr}.
Therefore $\mathcal R
= \mathcal S/\mathcal I$ is isomorphic to 
$\phi_*\bar{\mathcal R}^G$, and in particular
it is a Cox sheaf of type $\lambda$.
\end{proof}

\begin{remark}
 \label{rem:coxl}
We remark that by~\cite{dp}*{Definition~3.3},
the ring of global sections $\mathcal R(X)
= \Gamma(X,\mathcal S)/\Gamma(X,\mathcal I)$
is a Cox ring of type $\lambda$, and in particular
this implies that a Cox ring of type $\lambda$
for $X_{\kb}$ can be obtained from $\mathcal R(X)$ 
by a base change.
In particular, if $\lambda=\id_{\Cl(X_\kb)}$,
we get the usual Cox ring for $X_{\kb}$.
\end{remark}

\begin{lemma}
\label{quot}
Let $L/k$ be a Galois extension of fields 
with Galois group $G$. Let $V_1\subseteq V_2$ 
be $\kk$ vector spaces and let 
$\overline V_i = V_i\otimes_\kk L$.
Then $\overline V_1$ is $G$-invariant
and the homomorphism
\[
 j\colon V_2/V_1\to (\overline V_2/\overline V_1)^G
 \qquad
 v+V_1\mapsto v+\overline V_1
\]
is an isomorphism.
\end{lemma}
\begin{proof}
By hypothesis there is a $G$-equivariant
isomorphism $\overline V_2\simeq 
\overline V_1\oplus \overline T$ of $G$-invariant
vector spaces, where $T$ is obtained by 
completing a basis of $V_1$ to a basis
of $V_2$.
Since $\overline V_1\cap V_2 = \overline V_1^G = V_1$, 
the map $j$ is injective. To prove the surjectivity 
of $j$ let $v+\overline V_1\in (\overline V_2/\overline V_1)^G$,
that is $gv + \overline V_1 = v+\overline V_1$ for any $g\in G$,
or equivalently 
\[
 gv-v\in \overline V_1.
\]
Write $v = v_1 + t$ with
$v_1\in \overline V_1$ and $t\in \overline T$ and observe that
$gv-v\in \overline V_1$ implies $gt-t\in \overline V_1\cap \overline T = 0$, 
so that $t$ is $G$-invariant, that is $t\in V_2$.
Thus
\[
 v + \overline V_1 = t + \overline V_1 = j(t+V_1).
\]
\end{proof}

\section{Divisors on the generic fiber}
\label{sec:fib}
Let $X$ and $Y$ be normal
varieties defined over an algebraically closed 
field $\mathbb K$ of characteristic zero and 
satisfying~\eqref{assump}, and let
$\eta$ be the generic point of $Y$. 
In this section we are going to summarise
some results about the generic fiber 
$X_\eta :=X\times_Y\eta$ of a proper surjective morphism
$\pi\colon X\to Y$, whose very general fiber is irreducible.

First of all observe that the morphism 
$\imath\colon X_\eta\to X$
induces a pullback isomorphism
$\imath^*\colon \cc(X)\to\kk(X_\eta)$,
where $\kk = (\pi\circ\imath)^*(\cc(Y))\simeq \cc(Y)$.
We remark that the complementary
of the smooth locus $X^{\rm sm}$ has codimension 
at least two in $X$ and the same holds for the
generic fiber of the restriction $\pi_{|X^{\rm sm}}$
in $X_\eta$.
Therefore $\imath^*$ induces a surjective 
homomorphism
\begin{equation}
\label{i*div}
\WDiv(X)\to\WDiv(X_\eta).
\end{equation}
In what follows, by abuse of notation, we will use 
the same symbol $\imath^*$ for the above 
homomorphism, and we will denote by $\WDiv_\pi(X)$ 
its kernel.

\begin{proposition}
\label{imath}
The following hold:
\begin{enumerate}
\item the diagram 
\[
 \xymatrix{
  \cc(X)\ar[r]^-{\imath^*}\ar[d]^-{\div}
  & k(X_\eta) \ar[d]^-{\div}\\
  \WDiv(X)\ar[r]^-{\imath^*} 
  & \WDiv(X_\eta)
 }
\]
is commutative;
\item
if $D\in\WDiv(X)$ is effective on an open
subset $U\subseteq X$ then $\imath^*(D)$ 
is effective on the corresponding
open subset $U_\eta\subseteq X_\eta$;
\item
the group $\WDiv_\pi(X)$ is freely generated
by the prime divisors that do not 
dominate $Y$;
\item
the map~\eqref{i*div} induces a surjective homomorphism
$\Cl(X)\to\Cl(X_\eta)$, whose kernel $\Cl_\pi(X)$ is 
generated by the classes of divisors in $\WDiv_\pi(X)$;
\item
for any $D\in \WDiv(X)$
the pullback induces a map $\imath^*\colon\Osh_X(D)
\to\imath_*\Osh_{X_\eta}(\imath^*D)$.
\end{enumerate}
\end{proposition}

\begin{proof}
Recall that the generic fiber $X_\eta$ is limit
of the family of open subsets $\pi^{-1}(V)\subseteq X$,
where $V$ varies through the open subsets of $Y$.
Let us consider $V = \Spec(B)$, and let 
$U = \Spec(A)$ be an 
affine open subset of $\pi^{-1}(V)$.
The morphism $\pi|_U\colon U\to V$
is induced by an injective homomorphism 
of $\cc$-algebras, $B\to A$. Identifying $B$
with a subalgebra of $A$ we have that
the affine open subset $U_\eta\subseteq X_\eta$,
obtained by base change over $U$,
is the spectrum of the localization 
$S_B^{-1}A$, whose multiplicative
system is $S_B = B\setminus\{0\}$.
The pullback
\[
 \imath^*\colon\Osh_X(U)\to\imath_*\Osh_{X_\eta}(U_\eta)
\]
is thus defined on $U$ by the injection
$A\to S_B^{-1}A$. This shows that
a prime divisor $D$ defined by a prime ideal
$\mathfrak p\subseteq A$ survives in the 
generic fiber if and only if $\mathfrak p\cap B
= 0$, that is $D$ has non-empty intersection 
with $\pi^{-1}(V)$.
This proves (ii) and (iii).
In order to prove (i) recall that the
order of a rational function $f\in\cc(X)$
at $D$ is the length of the 
$\Osh_{X,\mathfrak p}$-module
$\Osh_{X,\mathfrak p}/\langle f\rangle$,
but the local rings $\Osh_{X,\mathfrak p}$
and $\Osh_{X_\eta,\mathfrak p}$
are isomorphic if $\mathfrak p\cap B = 0$.
In order to prove (iv), observe first that the
omomorphism
\[
\Cl(X) \to \Cl(X_\eta),
\qquad
[D] \mapsto [\imath^*(D)]
\]
is well defined by (i). Let us fix a divisor
$D\in\WDiv(X)$ such that $[D]$ is in the kernel 
$\Cl_\pi(X)$ of the above map.
By definition this implies that $\imath^*(D)$ is 
principal on $X_\eta$, so that  we can write 
$\imath^*(D)=\div(g)$, with $g\in\kk(X_\eta)$. 
Since $\imath^*\colon \cc(X)\to k(X_\eta)$
is an isomorphism we have $g=\imath^*(f)$, 
with $f\in\cc(X)$.
We conclude that 
\[
0=\imath^*(D)-\div(g) 
=\imath^*(D-\div (f))
\]
and in particular $D-\div(f)$ is a divisor
of $\WDiv_\pi(X)$, linearly equivalent to $D$.

Finally, to prove (v) let $f\in\Osh_X(D)(U)$ 
so that the divisor $\div(f)+D$ is effective 
on $U$. Then
\[
 \div(\imath^*(f)) + \imath^*(D) = \imath^*(\div(f)+D)
\]
is effective on $U_\eta$ by (ii).
\end{proof}

In what follows we will refer to the elements of $\WDiv_\pi(X)$ 
as {\em vertical divisors} and similarly to the elements
of $\Cl_\pi(X)$ as {\em vertical classes}. 

Let us remark that every pull-back divisor in $\WDiv(X)$
is vertical.
 In the next lemma we are going to prove that 
 in the case of principal divisor, also the converse
 is true.
In order to do that, let us denote by $\PDiv_\pi(X)$ 
the subgroup of $\WDiv_\pi(X)$ consisting
of the principal vertical divisors of $X$.

\begin{lemma}
\label{pvert}
Under the hypotheses above, the equality
$\PDiv_\pi(X) = \pi^*\PDiv(Y)$ holds. 
\end{lemma}
\begin{proof}
The inclusion $\pi^*\PDiv(Y)\subseteq
\PDiv_\pi(X)$ is obvious.
In order to prove the opposite
inclusion, let $D\in\PDiv_\pi(X)$
be a principal vertical divisor and let
$f\in\cc(X)$ be a rational function
such that $\div(f) = D$. By Proposition~\ref{imath}
we have
\[
 \div(\imath^*(f)) = \imath^*(D) = 0,
\]
and thus $\imath^*(f)$ must be constant,
being $X_\eta$ complete
by the properness 
hypothesis on $\pi$.
In particular $\imath^*(f)$ is an element
of $\kb\cap \kk(X_\eta)$, where $\kb$
is the algebraic closure of $\kk$.
By~\cite{laz}*{Example 2.1.12}
the following equality
\[
 \kb\cap\kk(X_\eta) = \kk
\]
holds, so that $\imath^*(f)\in\kk$.
In particular $f\in \pi^*(\cc(Y))$
and thus $D = \div(f)$ lies in
$\pi^*\PDiv(Y)$, which proves the
second inclusion.

\end{proof}

\section{Proof of Proposition~\ref{main:propo}}
\label{sec:div}
In this section we are going to prove
Proposition~\ref{main:propo} together 
with some related results. 

Let $X$ be a normal variety defined over 
a closed field $\cc$, such that $\Cl(X)$ is 
finitely generated and $\cc[X]^* = \cc^*$.
Let $K$ and $\mathcal S_K$ be as 
in Construction~\ref{con:closed}.
In what follows, whenever we need to keep 
trace of the degree of an element in 
$\Gamma(X,\mathcal S_K)_D$,
we will use the notation $f t^D$, where 
$f$ is in the Riemann Roch space of $D$.
If we denote by $\Frac_0(\Gamma(X,\mathcal S_K))$
the field of fractions of homogeneous sections 
having the same degree, we have the following.
\begin{lemma}
\label{iso_frac}
The map
$\mu_K\colon \Frac_0(\Gamma(X,\mathcal S_K)) 
\to \cc(X)$, defined by
$ft^D/gt^D \mapsto f/g$
is an isomorphism.
\end{lemma}
\begin{proof}
Since $\mu_K$ is a homomorphism of fields,  
we only need to show that it is surjective.
To this purpose, let us fix $h\in \cc(X)$. We can write
$\div(h) = A - B$, with $A$ and $B$ effective 
divisors in $\WDiv(X)$ without common support. 
Since the class map $\cl \colon K\to\Cl(X)$
is surjective, there exist
$D\in K$ and $g\in \Gamma(X,\mathcal S_K)_D$
such that $\div (g) + D = B$. The fraction
$hgt^D/gt^D$ is then an element in 
$\Frac_0(\Gamma(X,\mathcal S_K))$ whose
image via $\mu_K$ is $h$. 
\end{proof}
Let us suppose now that $X$ and $Y$ are
normal varieties having only constant 
invertible global sections and let 
$\pi\colon X\to Y$ be a proper surjective morphism 
whose very general fiber is irreducible.
Let $K$ be a finitely generated subgroup of
$\WDiv(X)$ such that the class map 
$K\to \Cl(X)$ is surjective.
In order to prove Proposition~\ref{main:propo} 
we are going to define the localization 
\[
\Gamma_\pi(X,\mathcal S_K)
=
S_\pi^{-1}\Gamma(X,\mathcal S_K),
\]
where $S_\pi$ is the multiplicative system consisting 
of the non-zero homogeneous elements 
whose degree $D\in K$ is such that 
$[D]\in\Cl_\pi(X)$. 

\begin{proof}[Proof of Proposition~\ref{main:propo}]
Let us denote by 
$\eta_K \colon\cc(Y)\to \Frac_0(\Gamma(X,\mathcal S_K))$
the composition $\mu_K^{-1}\circ \pi^*$.
We claim that it suffices to prove that
\[
 {\rm im}(\eta_K) = \Gamma_\pi(X,\mathcal S_K)_0.
\]
Indeed, by~\eqref{def:cox}, $\Gamma_\pi(X,\mathcal S_K)_0$
surjects onto $\R_\pi(X)_0$ and moreover the 
composition $\cc(Y)\to \R_\pi(X)_0$
is injective because the domain is a field.

In order to prove the inclusion ${\rm im}(\eta_K)\subseteq
\Gamma_\pi(X,\mathcal S_K)_0$,
let $s\in\cc(Y)$ and let $h := \pi^*(s)\in \cc(X)$. 
Write $\div(h) = A - B$, 
with $A$ and $B$ vertical, effective, 
with no common support.
Arguing as we did in the proof of Lemma~\ref{iso_frac}
we have that
\[
 \eta_K(s) 
 = \mu_K^{-1}(h)
 = hgt^D/gt^D
\]
where $D\in K$ is linearly equivalent to the 
vertical divisor $B$, so that the above quotient
is in $\Gamma_\pi(X,\mathcal S_K)_0$.

Viceversa let $ft^D/gt^D\in\Gamma_\pi(X,\mathcal S_K)_0$
be a degree zero homogeneous fraction.
The divisor $\imath^*(\div(f)+D)$ 
is effective by Proposition~\ref{imath}
and it is principal, being $[D]\in\Cl_\pi(X)$. 
Since $X_\eta$ is complete, 
$\imath^*(\div(f)+D) = 0$ or equivalently
$\div(f)+D$ is vertical. In the same way
one shows that also $\div(g)+D$ is vertical,
so that their difference $\div(f/g)$ is vertical
and principal. By Lemma~\ref{pvert}
the latter divisor is a pullback so that
$f/g = \pi^*(s)$ for some $s\in\cc(Y)$,
which proves the claim.
\end{proof}

Given $K$ as before and the map 
$\imath^*\colon \WDiv(X) \to \WDiv(X_\eta)$,
from now on for simplicity of notation
we will set $K_\eta :=  \imath^* (K)
\subseteq \WDiv(X_\eta)$.
By Proposition~\ref{imath} (v)
we have a morphism of sheaves of divisorial algebras
$\imath^*\colon\mathcal S_K
 \to\imath_*\mathcal S_{K_\eta}$ and
passing to global sections we obtain a 
homomorphism of rings
\begin{equation}
 \label{divmap}
 \imath^*\colon\Gamma(X,\mathcal S_K)
 \to\Gamma(X_\eta,\mathcal S_{K_\eta}).
\end{equation}

\begin{remark}
\label{inj}
If the subgroup $K$ does not contain vertical 
divisors, that is $K\cap\WDiv_\pi(X) = 0$,
then the restriction of 
$\imath^*\colon \WDiv(X) \to \WDiv(X_\eta)$ 
gives an isomorphism between $K$ and $K_\eta$.
In this case the map $\imath^*$ defined in~\eqref{divmap} 
is an injection since we have seen that $\imath^*$ induces also
an isomorphism between the fields of rational functions.
\end{remark}

\begin{proposition}
\label{global}
If the subgroup $K$ does not contain vertical 
divisors, then the map $\imath^*$ defined in~\eqref{divmap} 
extends to an isomorphism of $\cc(Y)$-algebras
\[
\imath^*\colon \Gamma_\pi(X,\mathcal S_K)\to
\Gamma(X_\eta,\mathcal S_{K_\eta}),
\qquad
\frac{f}{g}\mapsto
\frac{\imath^*(f)}{\imath^*(g)}.
\]
\end{proposition}

\begin{proof}
By Remark~\ref{inj} we already know that
the map~\eqref{divmap} is injective.
We are now going to prove that 
the image of an element in the multiplicative
system $S_\pi$ is invertible.
Let us fix $g\in S_\pi$, i.e. $g\in\Gamma(X,\mathcal S_K)_D$, 
and $[D]\in\Cl_\pi(X)$. Then $\imath^*(g)\in
\Gamma(X_\eta,\mathcal S_{K_\eta})_{\imath^*(D)}$,
where $\imath^*(D)$ is a principal divisor,
being its class trivial. Thus 
\[
 \imath^*(\div(g) + D)
 =
 \div(\imath^*(g)) + \imath^*(D) 
 = 
 0,
\]
where the first equality is by Lemma~\ref{imath}
and the second is due to the fact that
the generic fiber $X_\eta$ is complete,
being $\pi$ proper by hypothesis.
In particular $\imath^*(g)$ is invertible with
inverse $\imath^*(g^{-1})\in
\Gamma(X_\eta,\mathcal S_{K_\eta})_{\imath^*(-D)}$.
This shows that the map defined in the statement
is an injective homomorphism of $\cc(Y)$-algebras.

In order to prove the surjectivity,
it suffices to show that any homogeneous
$s\in\Gamma(X_\eta,\mathcal S_{K_\eta})_{\imath^*(D)}$,
with $D\in K$, is in the image.
At the level of rational functions we have 
$s=\imath^*(f)$, where $f\in \cc(X)$, with
\[
 \imath^*(\div(f) + D)
 =
 \div(\imath^*(f)) + \imath^*(D) 
 =
 E_\eta,
 \quad
 \text{effective on $X_\eta$}.
\]
If we denote by $E$ the Zariski closure of $E_\eta$ in $X$,
we have that $E$ is effective
too and $\imath^*(E) = E_\eta$.
Then the above formula implies
that $\div(f)+D = E + V$, where $V\in\WDiv_\pi(X)$. 
Write $V = A - B$, with
$A$ and $B$ effective. Let $B'\in K$ linearly
equivalent to $B$ and let
$h\in \Gamma(X,\mathcal S_K)_{B'}$ be such that
$\div(h)+B'=B$.
By the equality
\[
 \div(fh) + D + B'
 =
 E + A
\]
we deduce that $fh\in\Gamma(X,\mathcal S_K)_{D+B'}$
and thus  $\frac{fh}h\in\Gamma_\pi(X,\mathcal S_K)$
is a preimage of $s$.
\end{proof}

\section{Proof of Theorem~\ref{main:thm}}
\label{sec:prf}

In this section we are going to give
the proof of Theorem~\ref{main:thm}. 
In order to do that we first state and prove 
a couple of lemmas. Throughout the section
$\pi\colon X\to Y$ will be a proper surjection 
of normal varieties whose very general fiber 
is irreducible. From now on we also suppose 
that the subgroup $K\subseteq\WDiv(X)$ does
not contain vertical divisors, so that,
by Lemma~\ref{inj}, it is isomorphic to $K_\eta$. 
\begin{lemma}
\label{prim1}
If we denote by $K^0_\eta$ the kernel of the
surjection $K_\eta\to\Cl(X_\eta)$, we
have an isomorphism between 
$K^0_\eta/\imath^*(K^0)$ and $\Cl_\pi(X)$.
\end{lemma}
\begin{proof}
A diagram chasing in the following commutative 
diagram with exact rows and columns
establishes the claimed isomorphism
\[
\xymatrix{
&  &  & 0\ar[d] & \\
 & 0\ar[d] & 0\ar[d] & \Cl_\pi(X)\ar[d] & \\
0\ar[r] & K^0 \ar[r]\ar[d] & K\ar[d]\ar[r] &\Cl(X)\ar[d]\ar[r] & 0\\
0\ar[r] & K^0_\eta \ar[r]\ar[d] & K_\eta\ar[r]\ar[d] & \Cl(X_\eta)\ar[d]\ar[r] & 0.\\
& K^0_\eta/\imath^*(K^0)\ar[d] & 0 & 0 &\\
& 0
}
\]
\end{proof}

From now on we assume the group
${\rm Cl}_\pi(X)$ to be torsion-free.
Let us consider 
the following set of characters (which are sections
of the map $\div\colon \cc(X)\to\WDiv(X)$)
\[
 {\rm SecDiv}(K^0,\mathbb K(X)^*) 
 :=
 \{\mathcal X\colon K^0\to \cc(X)^*
 \mid \div\circ\mathcal X = \id\},
\]
and define ${\rm SecDiv}(K^0_\eta,
\kk(X_\eta)^*)$ in a similar way.
Let us define the map
\[
 \Phi
 \colon
 {\rm SecDiv}(K^0_\eta,\kk(X_\eta)^*)
 \to
 {\rm SecDiv}(K^0,\mathbb K(X)^*),
 \qquad
 \mathcal X_\eta\mapsto 
 {\imath^*}^{-1}\circ\mathcal X_\eta\circ \imath^*_{|K_0},
\]
and observe that the group $\Hom(\Cl_\pi(X),\kk^*)$
acts on ${\rm SecDiv}(K^0_\eta,\kk(X_\eta)^*)$
by multiplication.
\begin{lemma}
\label{chi}
The following hold:
\begin{enumerate}
\item
$\Phi$ is the quotient map by the action
of $\Hom(\Cl_\pi(X),\kk^*)$;
\item the pulback
$\imath^*\colon 
\mathcal S_K\to \mathcal S_{K_\eta}$
maps the ideal sheaf $\mathcal I$ to 
$\mathcal I_\eta$.
\end{enumerate}
\end{lemma}
\begin{proof}
We prove (i).
We first claim that $\Phi$ is
surjective. Indeed, let us fix a character 
$\mathcal X\in {\rm SecDiv}(K^0,\mathbb K(X)^*)$.
Observe that since ${\imath_{|K_0}^*}$
is an isomorphism on its image, there exists a 
unique homomorphism of groups 
$\varphi\colon  \imath^*(K^0) \to k(X_\eta)^*$
which makes the following diagram commute
\[
 \xymatrix{
  K^0\ar[r]^-{\mathcal X}\ar[d]^-{\imath^*_{|K_0}}
  & \cc(X)^*\ar[d]^-{\imath^*}\\
  \imath^*(K^0)\ar[r]^-{\varphi} 
  & k(X_\eta)^*.
 }
\]
Therefore for any divisor $D\in K^0$ we have
\[
\div(\varphi(\imath^*(D))) = \div(\imath^*(\mathcal X(D)))
= \imath^*(\div(\mathcal X(D))) = \imath^*(D),
\]
where the second equality follows from 
Proposition~\ref{imath}.
Now, by Lemma~\ref{prim1} and the assumption
that $\Cl_\pi(X)$ is torsion free, we deduce
that the subgroup $\imath^*(K^0)$ is primitive 
in $K^0_\eta$, and in particular $\varphi$ can 
be extended to a homomorphism $\mathcal X_\eta\colon 
K^0_\eta\to k(X_\eta)^*$.
Moreover the extension can be chosen 
in such a way that the 
equality $\div\circ\mathcal X_\eta = \id$ 
holds, so that $\mathcal X_\eta$
is an element of  ${\rm SecDiv}(K^0_\eta,\kk(X_\eta)^*)$.
By construction, $\Phi(\mathcal X_\eta) =\mathcal X$,
which proves the claim.

The map $\Phi$ is invariant for
the group action. Indeed, given 
$\gamma\in \Hom(\Cl_\pi(X),\kk^*)$
and $\mathcal X_\eta\in
{\rm SecDiv}(K^0_\eta,\kk(X_\eta)^*)$
we have $\Phi(\mathcal X_\eta)
= \Phi(\gamma\cdot\mathcal X_\eta)$
because $\gamma([D]) = 1$ for any
$D\in \imath^*(K^0)$.
Since the group action is clearly free,
in order to conclude the proof of (i)
we only need to 
show that the group $\Hom(\Cl_\pi(X),\kk^*)$ 
acts transitively on the fibers of $\Phi$.
Given two elements $\mathcal X_\eta$ 
and $\mathcal X_\eta'$ in the same fiber 
of $\Phi$, the homomorphism
\[
 K^0_\eta\to k(X_\eta)^*,
 \qquad
 D\mapsto\mathcal X_\eta(D)/\mathcal X_\eta'(D)
\]
is trivial on $\imath^*(K^0)$, and
thus by Lemma~\ref{prim1} it descends to a homomorphism
$\gamma\colon\Cl_\pi(X)\to\kk(X_\eta)^*$. Since
\[
 \div(\mathcal X_\eta(D)/\mathcal X_\eta'(D)) 
 =
 \div(\mathcal X_\eta(D)) - \div(\mathcal X_\eta'(D))
 =
 D - D = 0
\]
and $X_\eta$ is complete, we deduce that 
$\mathcal X_\eta(D)/\mathcal X_\eta'(D)
\in \kb^*\cap\kk(X_\eta)^*$. 
Moreover, by~\cite{laz}*{Example 2.1.12}, 
we have that $\kb^*\cap\kk(X_\eta)^*= k^*$, 
so that $\gamma\in {\rm Hom}(\Cl_\pi(X),\kk^*)$. 

We now prove (ii). Given a character 
$\mathcal X\in {\rm SecDiv}(K^0,\mathbb K(X)^*)$,
by Construction~\ref{con:closed}
we know that $\mathcal I$
is locally generated by the sections $1-\mathcal X(D)$,
for $D\in K^0$. By the surjectivity of $\Phi$
there exists a character 
$\mathcal X_\eta \in {\rm SecDiv}(K_\eta^0,\kk(X_\eta)^*)$
that makes the following diagram commute:
\[
 \xymatrix{
  K^0\ar[r]^-{\mathcal X}\ar[d]^-{\imath^*_{|K_0}}
  & \cc(X)^*\ar[d]^-{\imath^*}\\
  K^0_\eta\ar[r]^-{\mathcal X_\eta} 
  & k(X_\eta)^*.
 }
\]
Since $\imath^*(K_0) \subseteq K^0_\eta$ and,
by Construction~\ref{con:nonclosed}, 
$\mathcal I_\eta$ is locally generated by 
the sections $1-\mathcal X_\eta(D')$, 
for $D'\in K_\eta^0$, we get the statement.
\end{proof}

\begin{proof}[Proof of Theorem~\ref{main:thm}]
By Proposition~\ref{global}, Lemma~\ref{chi},
the characterization of Cox rings 
in~\cite{ADHL}*{Lemma 1.4.3.5}
and~\cite{dp}*{Construction 2.7},
we have a commutative diagram
with exact rows
\[
 \xymatrix{
  0\ar[r] &
  \Gamma_\pi(X,\mathcal I)\ar[r]\ar[d] &
  \Gamma_\pi(X,\mathcal S_K)\ar[r]\ar[d]^-\simeq_-{\imath^*} &
  \R_\pi(X)\ar[r]\ar[d]^-{\imath_R} &
  0\\
  0\ar[r] &
  \Gamma(X_\eta,\mathcal I_\eta)\ar[r] &
  \Gamma(X_\eta,\mathcal S_{K_\eta})\ar[r] &
  \R(X_\eta)\ar[r] &
  0\\
 }
\]
where $\Gamma_\pi(X,\mathcal I)$
is the localization of the ideal
$\Gamma(X,\mathcal I)$.
In particular the morphism
$\imath_R\colon\R_\pi(X) \to \R(X_\eta)$
is surjective, while the restriction of 
$\imath^*$ to $\Gamma_\pi(X,\mathcal I)$
is injective. 

Define the homomorphism of groups
\begin{equation}
 \label{eq:u}
 u\colon \Cl_\pi(X)\to\R_\pi(X)^*,
 \qquad
 [D]\to \mathcal {\imath^*}^{-1}(\mathcal X_\eta(D))
 +\Gamma_\pi(X,\mathcal I),
\end{equation}
where $D\in K^0_\eta$ and 
$\mathcal X_\eta(D)\in 
\Gamma(X_\eta,\mathcal S_{K_\eta})_{-D}$.
Let us show that $u$ is well
defined. Indeed, if $D'\in K^0_\eta$ is another
representative of the same class $[D]$,
by Lemma~\ref{chi} we have  
$D'-D = \imath^*(L)$ with $L\in K^0$,
so that
\begin{align*}
 {\imath^*}^{-1}(\mathcal X_\eta(D+i^*(L)))
 & = {\imath^*}^{-1}(\mathcal X_\eta(D)\cdot\mathcal X_\eta(\imath^*(L)))\\
 & = {\imath^*}^{-1}(\mathcal X_\eta(D))\cdot \mathcal X(L)\\
 & \equiv {\imath^*}^{-1}(\mathcal X_\eta(D))\mod{\Gamma_\pi(X,\mathcal I)}.
\end{align*}
Observe that the homomorphism $u$
satisfies the condition $u(w) \in \R_\pi(X)_{-w}^*$ 
for each $w\in \Cl_\pi(X)$. Moreover, given 
another morphism $u'\colon \Cl_\pi(X) \to \R_\pi(X)^*$
satisfying the same condition,
reasoning as in the proof of Lemma~\ref{chi} we
see that $u'/u\in\Hom(\Cl_\pi(X),\kk^*)$,
since $\R_\pi(X)_0^* \simeq\kk^*$.
Thus, by the same Lemma~\ref{chi},
$u'$ can be defined as in~\eqref{eq:u}
by taking the character 
$\mathcal X_\eta' = (u'/u)\cdot\mathcal X_\eta$.
We can now describe the kernel of 
$\imath_R$ as follows
\begin{align*}
 \ker(\imath_R)
 & = \{s+\Gamma_\pi(X,\mathcal I)\, :\, 
 \imath^*(s)\in\Gamma(X_\eta,\mathcal I_\eta)\}\\[2pt]
 & = \{{\imath^*}^{-1}(s_\eta)+\Gamma_\pi(X,\mathcal I)\, :\, 
 s_\eta\in\Gamma(X_\eta,\mathcal I_\eta)\}\\[2pt]
 & = \langle 1-{\imath^*}^{-1}(\mathcal X_\eta(D))
 +\Gamma_\pi(X,\mathcal I)\, :\, D\in K^0_\eta \rangle\\[2pt]
 & = \langle 1- u([D])\, :\, [D]\in \Cl_\pi(X)\rangle,
\end{align*}
which gives the claim.

\end{proof}

\begin{remark}
\label{prim}
Observe that if $\pi\colon X\to Y$ 
admits a rational section $\sigma$
and ${\rm Cl}(Y)$ is torsion free, 
then $\Cl_\pi(X)$ is torsion free
as well.
Indeed we can identify $\Cl_\pi(X)$ with the 
quotient $\WDiv_\pi(X)/\PDiv_\pi(X)$. So
let us take a divisor $V\in \WDiv_\pi(X)$ such that
$nV$ belongs to $\PDiv_\pi(X)$ for some integer
$n >1$.
Since $\PDiv_\pi(X)=\pi^*\PDiv(Y)$ we can write
$nV =  \pi^*D$, with $D\in\PDiv(Y)$ and hence
\begin{align*}
D & =
(\pi\circ\sigma)^*(D) =
\sigma^*(\pi^*D) =
\sigma^*(nV)  =
n\sigma^*(V).
\end{align*}
In particular $n\sigma^*(V)\in\PDiv(Y)$ and 
since we are assuming that $\Cl(Y)$ is
torsion free
we conclude that $\sigma^*(V)
\in\PDiv(Y)$. By applying $\pi^*$ to both
sides of the equation above we deduce
$nV=n\pi^*(\sigma^*(V))$ so that
$V=\pi^*(\sigma^*(V))$ holds since
$\WDiv(X)$ is free abelian.
In particular we conclude that
$V\in\pi^*\PDiv(Y) = \PDiv_\pi(X)$,
which proves the statement.
\end{remark}

\section{Very general fibers}
\label{sec:cor}
In this section we are going to apply the 
results of Theorem~\ref{main:thm}
in order to prove that, under an extra
hypothesis, it is indeed possible to recover 
the Cox ring of a very general fiber of 
$\pi\colon X\to Y$ from the Cox ring of $X$
and the vertical classes
(see Corollary~\ref{main:cor}).
In order to do that we need the following lemma.

\begin{lemma}
\label{lem:iso}
Let $X_i$, with $i\in\{1,2\}$ be a normal 
variety defined over an algebraically closed 
field $\kk_i$ of characteristic $0$. Assume 
that $\Cl(X_i)$ is finitely generated 
and that $\kk_i[X_i]^* = \kk_i^*$,
for any $i\in\{1,2\}$.
If there is an isomorphism of fields 
$\varphi\colon \kk_2\to\kk_1$
and an isomorphism of schemes 
$f\colon X_1\to X_2$ such that the 
following diagram commutes
\[
 \xymatrix{
  X_1\ar[r]^-f\ar[d] & X_2\ar[d]\\
  \Spec(\kk_1)\ar[r]^-{\varphi^*} & \Spec(\kk_2)
 }
\]
then $f$ induces an isomorphism of
graded rings $f^*\colon\R(X_2)\to \R(X_1)$,
such that $f^*|_{\kk_2} = \varphi$.
\end{lemma}
\begin{proof}
Observe that $f$ induces the pullback
isomorphism on the fields of rational
functions $\kk_2(X_2)\to \kk_1(X_1)$.
Given a prime divisor $D$ of $X_2$,
the restriction $D\cap X^\circ_2$ to
the smooth locus $X^\circ_2$ of $X_2$
is a Cartier non-trivial divisor, because
$X_2\setminus X^\circ_2$ has codimension
at least two by the normality of $X_2$.
Since $f$ is an isomorphism the pullback 
$f^*(D\cap X^\circ_2)$ is contained in the
smooth locus $X^\circ_1$ of $X_1$ and 
it has a unique closure by the 
normality of $X_1$. By linearity
the pullback map extends to an
isomorphism $f^*\colon\WDiv(X_2) \to 
\WDiv(X_1)$ of the groups of Weil divisors,
which maps principal divisors to principal 
divisors and thus gives also an isomorphism of 
divisor class groups $\Cl(X_2)\to\Cl(X_1)$.
By the above discussion, given a Weil
divisor $D$ of $X_2$ and an open subset
$U\subseteq X_2$, the pullback induces
an isomorphism of Riemann-Roch spaces
$\Gamma(U,\Osh_{X_2}(D))\to 
\Gamma(f^{-1}(U),\Osh_{X_1}(f^*D))$.
Thus, given a finitely generated subgroup 
$K\subseteq\WDiv(X_2)$ which surjects
onto $\Cl(X_2)$, the pullback gives an 
isomorphism of sheaves of divisorial algebras 
$\mathcal S_K\to f_*\mathcal S_{f^*K}$,
which induces an isomorphism of Cox sheaves.
By taking global sections we get an isomorphism
of Cox rings $f^*\colon\R(X_2)\to \R(X_1)$,
Finally observe that the last isomorphism
restricted to $k_2$ equals the restriction
of the isomorphism $\kk_2(X_2)\to \kk_1(X_1)$
and thus it coincides with $\varphi$.
\end{proof}
Let us go back to a morphism $\pi\colon X\to Y$
satisfying the hypotheses of Theorem~\ref{main:thm}
and let $X_\eta$ be the generic fiber of $\pi$. 
If we denote by $\bar X_\eta$ the base change 
$X_\eta \times_\kk \kb$, we have the
following.
\begin{corollary}
\label{main:cor}
Let $\pi\colon X\to Y$ satisfy the hypotheses 
of Theorem~\ref{main:thm}
and suppose in addition that the geometric divisor
class group $\Cl(\bar X_\eta)$ is
isomorphic to $\Cl(X_\eta)$. Then
the Cox ring of a very general fiber
of $\pi$ is isomorphic (as a graded ring)
to
\[
 \R_\pi(X)/\langle 1-u(w)\, :\, w\in\Cl_\pi(X)\rangle
 \otimes_\kk\kb
\]
where $u\colon \Cl_\pi(X)\to \R_\pi(X)^*$
is any homomorphism satisfying 
$u(w) \in \R_\pi(X)_{-w}^*$ for each 
$w$.
\end{corollary}

\begin{proof}
By~\cite{vi}*{Lemma~2.1} there exists a subset
$W\subseteq Y$ which is a countable intersection 
of non empty Zariski open subsets such that 
for each $b\in W$ there is an isomorphism
of rings $\cc\to\kb$ which induces an isomorphism
of schemes $X_b\to \bar X_\eta$.
Therefore by Lemma~\ref{lem:iso} the Cox ring 
of the very general fiber $X_b$ is isomorphic 
to the Cox ring $\R(\bar X_\eta)$ 
of the geometric generic fiber.
The isomorphism between $\Cl(\bar X_\eta)$ 
and $\Cl(X_\eta)$ implies that the former
can be generated by classes of divisors in 
$\WDiv_k(X_\eta)$. 
By Remark~\ref{rem:coxl}, the Cox ring $\R(\bar X_\eta)$ 
is obtained from $\R(X_\eta)$ by a base change, and
hence we can conclude by means of Theorem~\ref{main:thm}.
\end{proof}

\begin{remark}
We remark that the isomorphism 
of the corollary above is not an isomorphism
of graded algebras, since one of them is 
defined over $\cc$ while the other one
over $\kb$.
\end{remark}

\begin{remark}
\label{rem:mori}
Since by Corollary~\ref{main:cor}
the Cox ring of a very general fiber
can be described as a quotient of
a localization of the Cox ring of $X$,
if the latter is finitely generated, then
the former is finitely generated too.
In particular, if we can construct a
morphism $\pi\colon X\to Y$
satisfying the hypotheses of
the above corollary and such that
the Cox ring of a very general 
fiber is not finitely generated,
then we can conclude that 
the Cox ring of $X$ is not finitely 
generated. 
\end{remark}

\section{Blowing-ups of toric fiber spaces}
\label{sec:tor}
In this last section we apply our results
to the blowing-up of a toric fiber space
along a section, with the purpose
of producing new examples of 
varieties with non-finitely generated
Cox ring (see e.g.~\cites{GKK,GKK1,He,LU}).

\begin{construction}
\label{con:tor}
Let $\pi\colon X\to Y$ be a surjective 
toric morphism of normal toric varieties
which has connected fibers 
and such that the induced homomorphism of tori 
$\pi|_{T_X}\colon T_X\to T_Y$ is surjective
and $\Cl(Y)$ is torsion free.
If we denote by $X_0\subseteq X$ the Zariski
closure of the kernel of $\pi|_{T_X}$,
by~\cite{CLS}*{Eq. 3.3.6}
we have that $\pi^{-1}(T_Y) \simeq 
X_0\times T_Y$.
Let $x_0\in X_0$ be a general point and let 
$\phi\colon\tilde X\to X$ be the blowing-up
of $X$ along the Zariski closure of
$\{x_0\}\times T_Y$ via the above isomorphism.
Let $\tilde X_0$ be the preimage of 
$X_0$ via $\phi$, so that the restriction 
$\phi_{|_{\tilde X_0}}\colon\tilde X_0\to X_0$ is the blowing-up 
at $x_0$. We have a commutative diagram
\[
 \xymatrix{
  \tilde X_0\ar[r]\ar[d] & \tilde X\ar[d]^-{\tilde\pi}\\
  1_{T_Y}\ar[r] & Y
 }
\]
where the horizontal arrows are 
inclusions and $\tilde\pi$ denotes the
composition $\pi\circ\phi$.
\end{construction}

\begin{proposition}
\label{notfg}
If the Cox ring of $\tilde X_0$
is not finitely generated then the
same holds for the Cox ring of $\tilde X$.
\end{proposition}
\begin{proof}
By Remark~\ref{rem:mori},
it is enough to show that the toric morphism
$\tilde\pi \colon \tilde X\to Y$
satisfies the hypotheses of
Corollary~\ref{main:cor}.
By construction, the varieties $\tilde X$ 
and $Y$ are both normal, complete 
and $\tilde\pi$ is surjective.
The divisor class group $\Cl(\tilde X)$
is finitely generated, being $X$ toric.
All the fibers of $\tilde\pi$ over $T_Y$
are isomorphic to $\tilde X_0$ because
the point $x_0\in X_0$ is general,
and in particular they are connected
and irreducible. 
We now claim that the group
of vertical classes $\Cl_{\tilde\pi}(\tilde X)$
is torsion free.
Indeed, first of all observe that
the restriction $\pi|_{T_X}\colon T_X \to T_Y$
is a surjective homomorphism
of tori and thus it admits a section,
which gives a rational section
of $\pi\colon X\to Y$. Therefore,
by Remark~\ref{prim} (and the assumption
$\Cl(Y)$ free), we deduce that
$\Cl_\pi(X)$ is torsion-free. 
Moreover, if we denote by $\tilde X_\eta$
the generic fiber of $\tilde\pi$, we have 
that the exceptional divisor of
$\phi$ restricts to the exceptional
divisor of $\tilde X_\eta\to X_\eta$.
This gives an isomorphism 
$\Cl_{\tilde\pi}(\tilde X)
\simeq \Cl_\pi(X)$, and hence the claim.
Finally the pullback homomorphism $\Cl(\tilde X)
\to \Cl(\tilde X_0)$ is surjective
by Lemma~\ref{clsur}, and
this implies that the divisor 
class group of the generic geometric
fiber is isomorphic to the divisor class group
of $\tilde X_0$.
\end{proof}

\begin{lemma}
\label{clsur}
With the notation of Construction~\ref{con:tor},
the pullback homomorphism 
$\Cl(\tilde X)\to\Cl(\tilde X_0)$
is surjective.
\end{lemma}
\begin{proof}
It suffices to show that the pullback
$\Cl(X)\to \Cl(X_0)$ is surjective.
To this end, let $N_X$ and $N_Y$
be the lattices of one-parameter
subgroups of $X$ and $Y$ respectively
and let $\Sigma_X$ and $\Sigma_Y$ be 
their defining fans.
The morphism $\pi\colon X\to Y$ induces 
a surjective homomorphism 
$\alpha\colon N_X\to N_Y$ and, 
by~\cite{CLS}*{\S 3.3}, the 
fiber $X_0$ is the toric variety 
whose defining fan is 
\[
 \Sigma_0 
 :=
 \{\sigma\in\Sigma_X\, :\, 
 \sigma\subseteq \ker(\alpha)_{\mathbb Q}\}.
\]
Each torus invariant prime divisor $D_0$ of
$X_0$ corresponds to a one-dimensional
cone $\tau_0\in\Sigma_0$. 
Since $\Sigma_0$ is a subfan of $\Sigma_X$,
$\tau_0$ belongs also to $\Sigma_X$ 
and thus $D_0$ is the restriction of a torus invariant 
prime divisor of $X$, which proves the statement.
\end{proof}
We conclude with an example involving 
weighted projective spaces.
\begin{example}
Let $\pp(\mathbf a) $ be a weighted projective
space whose blowing-up at a general point
has non-finitely generated Cox ring
(see for instance~\cites{GKK,GKK1,He,LU}), 
and let $\Sigma_0\subseteq (N_0)_{\mathbb Q}$
be a defining fan for $\pp(\mathbf a)$.
Let $N := N_0\oplus\mathbb Z$. Given
$v\in N_0$ define the fan $\Sigma\subseteq
N_{\mathbb Q}$
whose maximal cones are
\[
 \{
 {\rm cone}(\sigma,(0,1)),
 {\rm cone}(\sigma,(v,-1))
 \, :\,
 \sigma\in \Sigma_0^{\rm max}\}.
\]
The projection $N\to \mathbb Z$ induces
a morphism $\pi\colon X(\Sigma)\to\mathbb P^1$,
whose general fiber is $\pp(\mathbf a)$.
Let $\tilde X$ be the blowing-up of
$X(\Sigma)$ along the image of a
rational section of $\pi$ passing 
through a general point of
$X_0$. By Proposition~\ref{notfg} 
we conclude that the Cox ring of 
$\tilde X$ is not finitely generated.
\end{example}

\end{document}